\documentclass[a4paper,12pt]{amsart}

\usepackage{amsthm,amssymb,amsmath,amscd}
\usepackage[mathscr]{euscript}
\usepackage{textcomp}

\theoremstyle{definition}
\newtheorem{definition}{Definition}

\theoremstyle{plain}
\newtheorem{theorem}[definition]{Theorem}
\newtheorem*{lemma*}{Lemma}

\newtheorem*{conj*}{Conjecture}

\theoremstyle{remark}

\title[Collapse of random groups]
{Collapse of random triangular groups:\\ a  closer look}
\author{Sylwia Antoniuk}
\address{Adam Mickiewicz University,
Faculty of Mathematics and Computer Science
ul.~Umultowska 87,
61-614 Pozna\'n, Poland}
\email{\tt antoniuk@amu.edu.pl}
\author{Tomasz \L{u}czak}
\address{Adam Mickiewicz University,
Faculty of Mathematics and Computer Science
ul.~Umultowska 87,
61-614 Pozna\'n, Poland}
\email{\tt tomasz@amu.edu.pl}
\author{Jacek \'{S}wi\c{a}tkowski}
\address{Instytut Matematyczny,
Uniwersytet Wroc{\l}awski,
pl. Grunwaldzki 2/4,
50-384 Wroc{\l}aw,
Poland}
\email{swiatkow@math.uni.wroc.pl}
\thanks{Jacek \'Swi\c atkowski partially
supported by the Polish Ministry of Science and Higher
Education (MNiSW), Grant N201 541738.}
\date{January 14, 2013}

\begin{document}
\maketitle

\begin{abstract}
The random triangular group $\Gamma(n,t)$ is a group given by a presentation $P=\langle S|R\rangle$, where $S$ is a set of $n$ generators
and $R$ is a random set of $t$ cyclically reduced words of length three. The asymptotic behavior of $\Gamma(n,t)$ is in some respects similar to
that of widely studied density random group introduced by Gromov. In particular, it is known that if $t\leq n^{3/2-\varepsilon}$
for some $\varepsilon > 0$, then with probability $1-o(1)$ $\Gamma(n,t)$ is infinite and hyperbolic, while for $t\geq n^{3/2+\varepsilon}$,
with probability $1-o(1)$
it
%$\Gamma(n,t)$
is trivial. In this note we  show that $\Gamma(n,t)$ collapses provided only that $t\geq C n^{3/2}$ for some constant $C>0$.
\end{abstract}

%\section[Introduction]{Introduction}

By $\langle S|R \rangle$ we denote a presentation $P$ of a group, where $S$ is the set of generators and $R$ the set of relations. A presentation $P$ is
a \textit{triangular presentation} if the set $R$ of relations consists of distinct cyclically reduced words of length three only
(over the alphabet $S\cup S^{-1}$  which consists of generators and their formal inverses).
A \textit{triangular group} is a group given by a triangular presentation.

In the paper we investigate the properties of groups given by presentations $P=\langle S|R \rangle$, where
$S$ is a set of $n$ generators and $R$ is a random subset of the set
$\mathcal{T}$ of all $N = 2n(4n^2-6n+3)\sim 8n^3$ distinct cyclically reduced words of length three, that is words of the form
$abc$, where $a\neq b^{-1}$, $b\neq c^{-1}$ and $c\neq a^{-1}$. Here we consider two models of a random
triangular group. In the random uniform model $\Gamma(n,t)$, the set of relations $R$ is chosen uniformly at random from the family of
all ${N\choose t}$ subsets of $\mathcal{T}$ of size $t$. In the binomial model $\Gamma(n,p)$, we select each element from
$\mathcal{T}$ independently with probability $p$. We shall study the asymptotic properties of these two models as
$n\rightarrow \infty$. In particular, we say that for a given function $t(n)$ [or $p(n)$] the group $\Gamma(n,t(n))$
[resp. $\Gamma(n,p(n))$] has a property a.a.s. (asymptotically almost surely) if the probability that the random group has this
property tends to 1 as $n\rightarrow \infty$. It should be mentioned that, as is well known in the theory of random structures,
the asymptotic behavior of $\Gamma(n,t)$ and $\Gamma(n,p)$ can be proved to be very similar provided $t\sim Np$. We  remark that $\Gamma(n,t)$
is basically identical with the model of random group introduced by \.{Z}uk \cite{Z2003}
who also pointed out the connection between his model
and a better known
construction  of random groups introduced by Gromov \cite{G1993}
(a more precise explanation of this relationship is given in \cite{O2005} and \cite{KK2012}).

%it turns out that the behavior of $\Gamma(n,t)$ is in many
%(but by no means all) cases similar to Gromov's model with
%the same density.

It is known (see \cite{Z2003} and Theorem 29 in \cite{O2005}) that
when the density $d=\frac{\ln t}{3\ln n}$ of $\Gamma(n,t)$ is smaller
than $1/2-\varepsilon$ for some $\varepsilon>0$, then a.a.s. $\Gamma(n,t)$ is  infinite
and hyperbolic, while for
$d > 1/2 + \varepsilon$  a.a.s. $\Gamma(n,t)$ collapses to the trivial group.  Here we study more precisely the case when the density is $1/2+o(1)$. As it was pointed out in \cite{O2005} part IV.a, this is the usual question after a talk on random groups.
The main results of this note  can be stated as follows.

\begin{theorem} \label{th1}
There exists an absolute constant $C_1 >0$ such that if $p\geq C_1 n^{-3/2}$ then a.a.s. $\Gamma(n,p)$ is trivial.
\end{theorem}

\begin{theorem} \label{th2}
There exists an absolute constant $C_2 >0$ such that if $t\geq C_2 n^{3/2}$ then a.a.s. $\Gamma(n,t)$ is trivial.
\end{theorem}

%\section{Proofs of Theorem \ref{th1} and Theorem \ref{th2}}

Our proof of Theorem \ref{th1} is based on the following fact concerning so called
random intersection graphs.  An \textit{intersection graph} is a graph on vertex set $V$ where each
vertex $v$ is assigned a set of features $W_v$ which is a subset of a ground set of features $W$.
Two vertices $v_1$ and $v_2$ are adjacent if they share a common feature, i.e. if $W_{v_1} \cap W_{v_2} \neq \emptyset$.
In a \textit{random intersection graph} $G_{n,m,\rho}$ the set of vertices $V$ is of size
$n$, the set of features $W$ is of size $m$ and for any vertex $v$ we assign each feature from $W$ to $v$ independently with
probability $\rho$.

Note that the probability that two vertices of  $G_{n,m,\rho}$  are adjacent is roughly
$$\hat{p} = 1-(1-\rho^2)^m\sim \rho^2 m.$$
If we denote
$m = n^{\alpha}$ then it turns out that for $\alpha > 1$ the random intersection graph $G_{n,m,\rho}$ behaves
in a similar manner as the random graph $G(n,\hat{p})$, in which each
pair of vertices is adjacent with probability $\hat{p}$, independently for each pair.
In 1960 Erd\H{o}s and R\'{e}nyi \cite{ER1960} in their seminal
paper on the evolution of random graphs proved that
when $n\hat{p}>c>1$, then a.a.s. in $G(n,\hat p)$ one can find
a unique giant component
which contains a positive fraction of all vertices
(more precisely, they showed it for an equivalent uniform model of the random graph, for details see \cite{ER1960} or \cite{JLR2000}).
  Berisch~\cite{B2007} showed that the fact that in $G_{n,m,\rho}$
the edges do not occur independently does not affect
very much the emergence of the giant component for the range of parameters
we are interested in. In particular, a special case of Theorem 1 in \cite{B2007}
is the following.

\begin{lemma*} \label{th3}
Let $G_{n,m,\rho}$ be a random intersection graph
with $m = n^{\alpha}$,
where $\alpha>1$, and $\rho^2m = \frac{\beta}{n}$.
Furthermore, for $\beta > 1$,
let $\gamma$ be the unique
solution to $\gamma = \exp(\beta(\gamma - 1))$ in the interval $(0,1)$.
Then a.a.s. the size of the largest component in $G_{n,m,\rho}$ is of order $(1+o(1))(1-\gamma)n$.

In particular, if $\beta\geq 1.42$ then a.a.s. $G_{n,m,\rho}$ contains a component of size at least 0.52n.
\end{lemma*}

\begin{proof}[Proof of Theorem \ref{th1}]
We shall show that the assertion holds whenever $C_1\ge 1.2$. 
Let us  partition the set of generators $S$ into two sets 
$S_1$ and $S_2$, where
$|S_1| = \lceil |S|/2 \rceil = \lceil n/2 \rceil$. We generate the set of relations $R$ of $\Gamma(n,p)$ in two stages.
Firstly, we consider  the relations  which contain exactly one element
from $S_1\cup S_1^{-1}$, and  select each such relations
 independently with probability $p=1.2 n^{-3/2}$. Then, we
select each of the remaining candidates for relations with the same
probability $p=1.2 n^{-3/2}$. We denote by $R_1$ the random set of relations generated in the first stage, and by $R_2$ the relations chosen in the second stage, so that  $R=R_1 \cup R_2$.

Consider an auxiliary random intersection graph $G_{n,m,\rho}$ with
vertex set $V=S_1\cup S_1^{-1}$, the set of objects $W = \{cd: c,d\in S_2\cup S_2^{-1}, c\neq d^{-1}\}$ and such that for
any $a\in V$ a feature $cd\in W$ is assigned to $a$ in $G_{n,m,\rho}$ if we have generated at least one
relation from the set $\{acd, cda, dac\}$. Note that here $m=n(n-1)$ and $\rho=1-(1-p)^3\geq p$. Therefore,
$\beta = \rho^2 m n \geq p^2 n^2(n-1) \geq 1.44 \frac{n-1}{n}$.

Using Lemma \ref{th3} we infer that a.a.s. $G_{n,m,\rho}$ contains a large component $L$ of at least $0.52n$ vertices. Note
however, that if two vertices $a,b$ are adjacent in $L$, then they share a common feature $cd\in W$, which implies that
$a=d^{-1}c^{-1}=b$ in $\Gamma(n,p)$. Moreover,
since $|L|>|S_1|$, for some $s\in S_1$, $L$ must contain both $s$ and $s^{-1}$. Consequently all elements of
$L$ are not only equal to $s$, but also satisfy the condition $s^2=e$.

Now consider relations generated in the second stage. Let $X$ count the number of elements $s$ from $L$ such that for any
$s', s''\in L$ the relation $ss's''$ does not belong to $R_2$. Then
$$\text{Pr}(X > 0) \leq \mathbb{E}X \leq 0.52n(1-p)^{0.25n^2} \leq 0.52n\exp(-0.36\sqrt{n}),$$
and tends to 0 as $n\rightarrow\infty$. Hence, a.a.s. $L$ contains three elements $s,s',s''$ such that $ss's''\in R_2$. However, since
all elements from $L$ are equal, we conclude that $s^3=e$. This, together with the condition that $s^2=e$, implies that a.a.s.
all generators or inverses of generators contained in $L$ are equivalent to the identity.

In a similar way one can estimate the number of generators $s$ not contained in $L$, for which no triple $ss's''$ belongs to $R_2$,
where $s',s''\in L$. Again, the probability that there is at least one such generator tends to 0 as $n\rightarrow \infty$,
therefore a.a.s. each generator from $S$ is equivalent to the identity and the assertion follows.
\end{proof}

\begin{proof}[Proof of Theorem \ref{th2}]
Note that the property that $\langle S|R \rangle$ is trivial
is a monotone property of the set of relations~$R$. 
Therefore, to show that the asymptotic behavior of
$\Gamma(n,t)$ in terms of generating a trivial group is similar to the asymptotic behavior of $\Gamma(n,p)$, one can use
exactly the same method as in the proof of asymptotic equivalence for monotone properties of the uniform random graph model
$G(n,M)$ and the binomial random graph model $G(n,p)$
(see \cite{JLR2000}).
\end{proof}

The result proved above together with the following conjecture
(stated for $\Gamma(n,t)$, but having the analogue for $\Gamma(n,p)$),
provide a sharp threshold for collapsibility of a triangular random group.

\begin{conj*}
There exists a constant $c>0$ such that for every $t\le cn^{3/2}$
a.a.s. $\Gamma(n,t)$ is infinite.
\end{conj*}

As a matter of fact, we conjecture that for small $c$ and
for $t<cn^{3/2}$ a.a.s.
$\Gamma=\Gamma(n,t)$ is an aspherical group, i.e.
it has no reduced spherical van Kampen diagram.
In this case the presentation complex
$C$ of this group is aspherical,
and hence it is a classifying space for this group, which has
the following consequences:
\begin{itemize}
\item the Euler characteristic $\chi(\Gamma)=\chi(C)=1-n+t$;
\item $\Gamma$ is torsion-free (see \cite{Br}, p. 187).
\end{itemize}

\noindent
Since a group $\Gamma$ with $\chi(\Gamma)\ne1$ is nontrivial (and this is true
when $t\sim n^{3/2}$),
%and in view of the earlier mentioned hyperbolicity result of \.Zuk,
the above ``asphericity conjecture''
implies Conjecture (for $t\ll n^{3/2}$ Conjecture is true by the earlier mentioned
result of \.Zuk).

As a circumstantial evidence supporting the above ``asphericity conjecture''
we remark that it is not hard to show that for small $c$ and $t\le cn^{3/2}$ a.a.s.
$\Gamma(n,t)$ has no reduced spherical van Kampen diagram with cells corresponding to pairwise distinct relations. Moreover, for
every constant $A$, a.a.s.  $\Gamma(n,t)$ has no reduced spherical van Kampen diagram with fewer than $A$ cells.

%Note however that for the range of $p$ we consider $\Gamma(n,t)$ is not ``uniformly'' hyperbolic (by this we mean that even if $\Gamma(n,t)$  is a.a.s. hyperbolic the constant in the definition of hyperbolicity cannot be bounded from above by a constant $B$ independent of $n$, although, possibly, could be bounded by some function of $n$). Consequently, local-to-global principle typically used in such cases does not  apply for this range  of $t$.

\medskip
Finally, let us mention that one can easily use the same argument
to show that for every $\ell$ there exists a constant
$a_\ell$ such that a.a.s. as $m\to\infty$
the random group given by presentations $P=\langle S|R \rangle$,
where $S$ is a generating set of cardinality $m$, and where $R$ is a random subset of the set of
cyclically reduced words of length $\ell$,
collapses to the trivial group provided only
$|R|\ge a_\ell (2m-1)^{\ell/2}$. If $\ell$ is a prime one can just mimic
our argument; in the remaining cases there is a tiny technical problem
to deal with rare cyclical words which can be easily solved. This
result matches nicely a theorem of Gady Kozma, mentioned
in~\cite{O2010}, which concerns the Gromov model
when  $m$ is fixed and
$\ell\to\infty$.

\section*{Acknowledgment}
We would like to thank Katarzyna Rybarczyk for her valuable remarks on random intersection graphs.

\bibliographystyle{plain}

\end{document}